\documentclass[11pt,leqno]{article}
  
  \usepackage{amsfonts}\usepackage{amssymb}
  \usepackage{amsthm}\usepackage{amsmath}\usepackage{stmaryrd}
  \usepackage{cite}
 
	\newcommand{\R}{\mathbb R}
  \newcommand{\sn}{S^{n-1}}

	\newcommand{\cP}{{\mathcal P}} 
	\newcommand{\cK}{{\mathcal K}} 
	 
	\newcommand{\cN}{{\mathcal N}}

  \newcommand{\cQ}{{\mathcal Q}}
   \newcommand{\cM}{{\mathcal M}}
\newcommand{\C}{{\textnormal C}}

  \newcommand{\gln}{\operatorname{GL}(n)}
	
  \newcommand{\sln}{\operatorname{SL}(n)}

	\newcommand{\oz}{\operatorname{\rm Z}}

	\newcommand{\lin}{\operatorname{lin}}

	\newtheorem{theorem}{Theorem}
	\newtheorem{lemma}[theorem]{Lemma}   
	\newtheorem{corollary}[theorem]{Corollary} 

\begin{document}                                                 
\title{Valuations and Surface Area Measures}                                 
\author{Christoph Haberl and Lukas Parapatits}
\date{}
\maketitle

\begin{abstract}
\noindent
We consider valuations defined on polytopes containing the origin which have measures on the sphere as values.
We show that the classical surface area measure is essentially the only such valuation which is $\sln$ contravariant of degree one.
Moreover, for all real $p$, an $L_p$ version of the above result is established for $\gln$ contravariant valuations of degree $p$.
This provides a characterization of the $L_p$ surface area measures from the $L_p$ Brunn-Minkowski theory. \\

\end{abstract}                

\maketitle  

\section{Introduction}\label{intro}

A valuation is a map $\mu:\cQ^n\to\langle A,+\rangle$ defined on a set $\cQ^n$ of subsets of $\R^n$ with values in an abelian semigroup such that
\[\mu(P\cup Q)+\mu(P\cap Q)=\mu (P)+\mu (Q)\]
whenever the sets $P$, $Q$, $P\cup Q$, $P\cap Q$ are contained in $\cQ^n$. Valuations were the critical ingredient in Dehn's solution of Hilbert's Third Problem and played a prominent role in geometry thereafter. Especially over recent years the theory of valuations witnessed an explosive growth (see e.g.\cite{AleBerSchu,Alesker99,Alesker01,Bernig08,BerFu10,fu06,Klain97,HL06,Hab10amj,Lud05,Lud03,Hab11,LR10,Lud11,Lud10,Schduke,schuster08,SchWan10,ParSchu}). For example, it turned out that basic objects in convex geometric analysis can be characterized as valuations which are compatible with a subgroup of the affine group. Moreover, new geometric insights gained from such classification results led to strengthenings of various affine isoperimetric and Sobolev inequalities (see \cite{HS09jdg,HS09jfa,CLYZ,LutYanZha02jdg,LutYanZha00jdg}). 

In this paper we classify measure valued
valuations which are compatible with the general linear group. We will show that the only non-trivial examples of such valuations are surface area measures and their $L_p$ analogs.

Surface area measures are a fundamental concept in the theory of convex bodies, i.e.\! nonempty compact convex subsets of $\R^n$. Given a convex polytope $P\subset\R^n$, its surface area measure $S(P,\cdot)$ is the Borel measure on the Euclidean unit sphere $\sn\subset\R^n$ which is given by
\begin{equation}\label{def surface area}
S(P,\cdot)=\sum_{u\in\cN(P)} V_{n-1}(F_u)\delta_{u}.
\end{equation}
Here, $\cN(P)$ denotes the set of all unit facet normals of $P$, $V_{n-1}(F_u)$ stands for the $(n-1)$-dimensional volume of the facet corresponding to $u$, and $\delta_{u}$ is the Dirac probability measure supported at $u$. Surface area measures can actually be associated with arbitrary convex bodies. We write $\cK^n$ for the space of convex bodies in $\R^n$ equipped with Hausdorff distance. For a body $K\in\cK^n$, its surface area measure $S(K,\cdot)$ is simply the weak limit of the measures $S(P_k,\,\cdot\,)$, where $P_k$ is some sequence of polytopes converging to $K$.

Surface area measures are the central object of a cornerstone of the classical Brunn-Minkowski theory: The Minkowski problem. It asks which measures on the Euclidean unit sphere are surface area measures of convex bodies. The answer to this question had a huge impact on convex geometry, geometric tomography, differential geometry, and elliptic partial differential equations (see e.g. \cite{Gar95,Sch93,24}). 

The first main result of this paper is a characterization of the surface area measure.
We will consider valuations which are defined on polytopes containing the origin and which have measures on the sphere as values.
We show that the surface area measure is essentially the only such valuation with a certain compability with the special linear group.
What we mean by compatibility is contained in the following definition. Let $G$ be a subgroup of the general linear group $\gln$ and denote by $\cM(\sn)$ the space of finite positive Borel measures on $\sn$. Suppose that $p\in\R$. A map $\mu:\cQ^n\to\cM(\sn)$ defined on $\cQ^n\subset\cK^n$ is called  $G$ contravariant of degree $p$ if
\[\int_{\sn}f\,d\mu(\phi P,\cdot)=|\det\phi|\int_{\sn}f\circ\phi^{-t}\,d\mu(P,\cdot)\]
for every map $\phi\in G$, each $P\in\cQ^n$ with $\phi P\in\cQ^n$, and every continuous $p$-homogeneous function $f:\R^n\backslash\{o\}\to\R$.
Here, $\phi^{-t}$ denotes the transpose of the inverse of $\phi \in \gln$.
We remark that the concept of $G$ contravariance is simply the behavior of mixed volumes (and their $L_p$ analogs) under the action of the general linear group (see Section \ref{not}).

We are now in a position to state our first main result. Throughout this article we work in $\R^n$ with $n\geq3$. Convex polytopes in $\R^n$ which contain the origin are denoted by $\cP_o^n$.

\begin{theorem}\label{main1} A map $\mu:\cP_o^n\to\cM(\sn)$ is an $\sln$ contravariant valuation of degree 1 if and only if there exist constants $c_1,c_2,c_3,c_4\in\R$ with $c_1,c_2\geq0$ and $c_1+c_3\geq0$, $c_2+c_4\geq 0$ such that
\[\mu(P,\cdot)=c_1S(P,\cdot)+c_2S(-P,\cdot)+c_3S^*(P,\cdot)+c_4S^*(-P,\cdot)\]
for every $P\in\cP_o^n$.
\end{theorem}
The measure $S^*(P,\cdot)$ is defined similarly to the surface area measure of $P$ but the summation in (\ref{def surface area}) ranges only over those facets in $\cN(P)$ which do not contain the origin. An immediate consequence of the above result is the following classification of measure valued valuations defined on {\it all} convex bodies.

\begin{corollary} A map $\mu:\cK^n\to\cM(\sn)$ is a weakly continuous, translation invariant, and $\sln$ contravariant valuation of degree 1 if and only if there exist nonnegative constants $c_1,c_2\in\R$ such that
\[\mu(K,\cdot)=c_1S(K,\cdot)+c_2S(-K,\cdot)\]
for every $K\in\cK^n$.
\end{corollary}

Schneider \cite{Sch75} previously obtained a classification of rotation contravariant valuations of degree 1 under the additional assumption that they are defined locally.

As explained before, surface area measures lie at the very core of the Brunn-Minkowski theory. Based on Firey's $L_p$ addition for convex bodies, Lutwak \cite{Lut93b, Lut96} showed that the classical Brunn-Minkowski theory can be extended to an $L_p$ Brunn-Minkowski theory. The importance of this new $L_p$ theory is reflected for example in the fact that $L_p$ inequalities almost invariably turn out to be stronger than their
classical counterparts. Since Lutwak's seminal work, this $L_p$ Brunn-Minkowski theory evolved enormously (see e.g.
\cite{ChoWan06, Gardner:Giannopoulos, HS09jdg,HS09jfa,
Lud05, LR10, LutYanZha00duke, LutYanZha00jdg,
148, SchWer04,
Sta03,YasYas06,WerYe09}).

Let $p\in\R$. The analog of the surface area measure in the $L_p$ Brunn-Minkowski theory is defined as follows. For a convex polytope $P\in\cP_o^n$, the $L_p$ surface area measure $S_p(P,\cdot)\in\cM(\sn)$ is given by
\[S_p(P,\cdot)=\sum_{u\in\cN^*(P)} h(P,u)^{1-p}V_{n-1}(F_u)\delta_{u},\]
where $\cN^*(P)$ denotes the set of unit facet normals of $P$ corresponding to facets which do not contain the origin and $h(P,\cdot)$ is the support function of $P$ (see Section \ref{not} for the precise definition).

Finding necessary and sufficient conditions for a measure to be the $L_p$ surface area measure of a convex body is one of the major problems in modern convex geometric analysis. Consequently, this $L_p$ analog of the Minkowski problem has been studied intensively (see e.g. \cite{HabLutYanZha09,ChoWan06,Lut93b,Sta03}). Solutions to the $L_p$ Minkowski problem were crucial for the proofs of affine versions of the P\'olya-Szeg\"o principle and new affine Sobolev inequalities (see \cite{HS09jdg,HS09jfa,CLYZ,LutYanZha02jdg,LutYanZha00jdg}).

The following theorem provides a characterization of $L_p$ surface area measures for all $p\neq1$.

\begin{theorem}\label{main5} Let $p\in\R\backslash\{1\}$. A map $\mu:\cP_o^n\to\cM(\sn)$ is a $\gln$ contravariant valuation of degree $p$ if and only if there exist nonnegative constants $c_1,c_2\in\R$ such that
\[\mu(P,\cdot)=c_1S_p(P,\cdot)+c_2S_p(-P,\cdot)\]
for every $P\in\cP_o^n$.
\end{theorem}

Since the $L_p$ Brunn Minkowski theory is based on an addition which makes sense only for $p\geq 1$, most of the $L_p$ concepts are restricted to such $p$'s. However, the above theorem reveals that the concept of $L_p$ surface area measures is independent of $p$ in a very natural way.

For positive $p$, we will actually prove a stronger version of Theorem \ref{main5}. It will be shown in Theorem \ref{main4} that for such $p$'s the $L_p$ surface area measure can actually be characterized as an $\sln$ contravariant valuation.  Moreover, Theorem \ref{main3} will show that $L_p$ surface area measures are characterized as valuations which are $\sln$ contravariant of degree $p$ for {\it all} $p\in\R$ provided that their images are discrete.

Recently, the next step in the evolution of the Brunn-Minkowski theory towards an Orlicz Brunn-Minkowski theory has been made (see e.g. \cite{LutYanZha09jdg,LutYanZha09badv,HabLutYanZha09,LR10}). Whereas some elements of the $L_p$ Brunn-Minkowski theory have been generalized to an Orlicz setting, the Orlicz analog of the surface area measure is still unknown. This question actually motivated the axiomatic characterization of $L_p$ surface area measures obtained in this article. Since characterizing properties of $L_p$ surface area measures are now identified, they can possibly lead to the correct notion of Orlicz surface area measures.


\section{Notation and Preliminaries}\label{not}

In this section we collect the necessary definitions and facts about convex bodies. Excellent references for the theory of convex bodies are the books by Gardner \cite{Gar95}, Gruber \cite{Gruber:CDG}, and Schneider \cite{Sch93}.

We write $\R_+$ for the set of positive real numbers. Given two vectors $x,y\in\R^n$ we write $x\cdot y$ for their standard Euclidean product. The Euclidean length of a vector $x\in\R^n$ is denoted by $|x|$. If $x\in\R^n$ is not equal to the zero vector, then we set
$$\langle x\rangle=\frac{x}{|x|}\,.$$
The canonical basis vectors of $\R^n$ are denoted by $e_1,\ldots,e_n$. The standard simplex $T^n\subset\R^n$ is the convex hull of the origin and the canonical basis vectors $e_1,\ldots,e_n$. We denote by $T'$ the intersection $T^n\cap e_1^{\bot}$ where $e_1^{\bot}$ stands for the hyperplane through the origin orthogonal to $e_1$. If $p$ is positive, then we write $C_p^+(\R^n)$ for the space of nonnegative, continuous, $p$-homogeneous functions from $\R^n$ to $\R$.

For the definition of $L_p$ surface area measures we already used the notion of support functions. The precise definition is as follows. Given a convex body $K\in\cK^n$, its support function is defined as
\[h(K,x)=\max\{x\cdot y:\,\,y\in K\},\qquad x\in\R^n.\]

It follows from the inclusion-exclusion principle that a valuation $\mu:\cP_o^n\to\cM(\sn)$ is uniquely determined by its values on $n$-dimensional simplices having one vertex at the origin and its value on $\{o\}$ (see \cite{Par10a} for a short proof). If $\mu$ is in addition $\sln$ contravariant of degree $p$, then the uniqueness part of Riesz's representation theorem implies that -- beside its behavior at the origin -- $\mu$ is uniquely determined by its values on the simplices $sT^n$ with $s>0$. We summarize this in the following lemma.
\begin{lemma}\label{ext}
Let $p\in\R$. A valuation $\mu:\cP_o^n\to\cM(\sn)$ which is $\sln$ contravariant of degree $p$ is uniquely determined by its values on positive multiples of the standard simplex $T^n$ and its value on $\{o\}$.
\end{lemma}

A measure $\mu\in\cM(\sn)$ is called continuous if singletons have $\mu$-measure zero. We call it discrete, if there exists a countable set $N\subset\sn$ such that $\mu(\sn\backslash N)=0$. The set of all discrete members of $\cM(\sn)$ is denoted by $\cM^d(\sn)$. Note that for $\mu,\,\nu\in\cM^d(\sn)$ we have
\begin{equation}\label{eqivmeas}
\mu=\nu\Longleftrightarrow \mu(x)=\nu(x)\textnormal{ for all } x\in\sn.
\end{equation}
Here we used the convention $\mu(x):=\mu(\{x\})$ for $x\in\sn$. For every $\mu\in\cM(\sn)$ there exists a unique pair consisting of a continuous measure $\mu^c$ and a discrete measure $\mu^d$ such that
\begin{equation}\label{mudec}
\mu=\mu^c+\mu^d.
\end{equation}
Let $\mu:\cP_o^n\to\cM(\sn)$ be given. If $\mu$ is a valuation, so is $\mu^d$. Indeed, since the valuation property has to be checked only for points by (\ref{eqivmeas}), the assertion directly follows from the decomposition (\ref{mudec}). Note that if $\mu$ is $\sln$ contravariant of degree $p$, then by the uniqueness part of Riesz's representation theorem and the transformation behavior of image measures we have
\[\int_{\sn}f d\mu(\phi P,\cdot)=\int_{\sn} f\circ\phi^{-t} d\mu(P,\cdot)\]
for all $p$-homogeneous extensions of bounded Borel measurable functions $f:\sn\to\R$. In particular, the last relation holds for indicator functions of points. This yields
\begin{equation}\label{eqn1}
\mu(\phi P,\langle x\rangle)|x|^{-p}=\mu(P,\langle \phi^t x\rangle)|\phi^t x|^{-p}
\end{equation}
for all $P\in\cP_o^n$, $x\in\R^n\backslash\{o\}$, and all $\phi\in\sln$. Note that this together with (\ref{mudec}) implies the $\sln$ contravariance of degree $p$ of $\mu^d$ provided that $\mu$ is $\sln$ contravariant of degree $p$.

Surface area measures and their $L_p$ analogs were already defined in the introduction. It will be convenient for us to write
\[S_p(P,x)=S_p(P,\langle x\rangle)|x|^{-p}\qquad \textnormal{for}\qquad x\in\mathbb{R}^n\backslash\{o\}.\]
Similar conventions will apply to the measures $S^*(P,\cdot)$ and $S^o(P,\cdot):=S(P,\cdot)-S^*(P,\cdot)$. 
The following lemma guarantees that surface area measures and their $L_p$ analogs are $\gln$ contravariant valuations.

\begin{lemma} The measures $S$, $S^*$, $S^o$ and the $L_p$ surface area measures $S_p$ for $p\neq 1$ are $\gln$ contravariant valuations of degree $p$ on $\cP_o^n$.
\end{lemma}

\begin{proof}First, assume that $p\neq 1$. Given a $p$-homogeneous function $f:\R^n\backslash\{o\}\to\R$, define for $v\in\sn$,
\[
f^*(t,v)=\left\{\begin{array}{ll}
0&t\leq 0,\\
t^{1-p}f(v)&t>0.
\end{array}\right.
\]
It was shown in \cite{Par10a} that for such functions $f^*:\R\times\sn\to\R$ the expression
\[\oz_f P = \sum_{v\in\cN(P)}V_{n-1}(F(P,v))f^*(h(P,v),v)\]
is a real-valued valuation. Here, $F(P,v)$ denotes the facet of $P$ with outer unit normal vector $v$.
Let $\omega$ be a Borel set on $\sn$ and take $f(x)=\mathbb{I}_{\omega}(x/|x|)|x|^p$, where $\mathbb I_\omega$ denotes the indicator function of $\omega$.
Then we have $S_p(P,\omega)=\oz_f P$ and we see that $S_p$ is a measure valued valuation. For $\phi\in\gln$ and $v\in\sn$ we clearly have
\[v\in\cN(P)\Longleftrightarrow\langle\phi^{-t}v\rangle\in\cN(\phi P)\] 
as well as
\[V_{n-1}(F(\phi P,\langle \phi^{-t} v\rangle))=\|\phi^{-t} v\|\,|\det\phi|V_{n-1}(F(P,v)).\]
The transformation behavior of the support function with respect to the general linear group and the homogeneity of $f^*$ yield
\begin{eqnarray*}
f^*(h(\phi P, \langle \phi^{-t} v\rangle),\langle \phi^{-t} v\rangle)&=& f^*(h(P,\phi^t \langle \phi^{-t} v\rangle),\langle \phi^{-t} v\rangle)\\
&=&f^*(\|\phi^{-t}v\|^{-1}h(P,v),\|\phi^{-t}v\|^{-1}\phi^{-t}v)\\
&=&\|\phi^{-t}v\|^{-1}f^*(h(P,v),\phi^{-t}v).
\end{eqnarray*}
Therefore we obtain that
\[V_{n-1}(F(\phi P,\langle \phi^{-t} v\rangle))f^*(h(\phi P,\langle \phi^{-t} v\rangle),\langle \phi^{-t} v\rangle)=|\det\phi|V_{n-1}(F(P,v))f^*(h(P,v),\phi^{-t}v).
\]
This immediately implies the $\gln$ contravariance of degree $p$ of $S_p$.
For $p=1$, the same proof yields the desired properties for $S^*$.
By changing the definition of $f^*$ to
	\[
	f^*(t,v)=\left\{\begin{array}{ll}
	0&t < 0,\\
	f(v)&t \geq 0
	\end{array}\right.
	\]
we obtain these properties for the surface area measure $S$.
Therefore we also have them for $S^o=S-S^*$.
\end{proof}
For $p>0$, the $L_p$ cosine transform of a signed finite Borel measure $\mu$ on $S^{n-1}$ is defined by
\[\C_p\mu(u)=\int_{\sn}|u\cdot v|^p\,d\mu(v),\qquad u\in\sn.\]
We need the following injectivity result: For a finite signed Borel measure $\mu$ on $\sn$
\begin{equation}\label{inject}
\C_1\mu=0\Longrightarrow \mu(\omega)=\mu(-\omega) ,
\end{equation}
for each Borel set $\omega$ on $S^{n-1}$. 

In the proof of our classification results for positive $p$ we will make use of known characterizations of function valued valuations. Therefore, we need a translation of $\sln$ contravariance to such valuations. For positive $p$, a function $\oz:\cP_o^n\to C_p^+(\R^n)$ is called $\sln$ contravariant if $\oz(\phi P)(x)=\oz P(\phi^{-1} x)$ for all $P\in\cP_o^n$, each $\phi\in\sln$, and all $x\in\R^n$. The next two results were established in \cite{Hab11} and \cite{Par10a}, respectively.

\begin{theorem}\label{funcchar1}If $\oz:\cP_o^n\to\langle C_1^+(\R^n),+\rangle$ is an even $\sln$ contravariant valuation, then there exist constants $c_1,c_2\in\R$ such that
\[\oz P =\C_1 \left( c_1S(P,\cdot)+c_2S^*(P,\cdot) \right) \]
for every $P\in\cP_o^n$.
\end{theorem}

\begin{theorem}\label{funcchar2}Let $p\in\R_+\backslash\{1\}$. If $\oz:\cP_o^n\to\langle C_p^+(\R^n),+\rangle$ is an even $\sln$ contravariant valuation, then there exists a constant $c\in\R$ such that
\[\oz P =\C_p \left( cS_p(P,\cdot) \right) \]
for every $P\in\cP_o^n$.
\end{theorem}

As announced in the introduction, we briefly describe where the notion of $G$ contravariance comes from. The basis of the $L_p$ Brunn-Minkowski theory is the following addition for convex bodies. Let $p\geq 1$ and suppose that $P,Q\subset\cP_o^n$ contain the origin in their interiors. For $\varepsilon >0$ there exists a unique convex body $P+_p \varepsilon\cdot Q$ such that
\[h(P+_p \varepsilon\cdot Q,\cdot)^p=h(P,\cdot)^p+\varepsilon h(Q,\cdot)^p.\]
Using this addition, Lutwak \cite{Lut93b} extended the classical case $p=1$ in order to prove that for the volume $V$ and all $p\geq 1$,
\[\lim_{\varepsilon\to0^+}\frac{V(P+_p\varepsilon\cdot Q)-V(P)}{\varepsilon}=\frac{1}{p} \int_{\sn}h(Q,u)^p\,dS_p(P,u).\]
This limit is called the $L_p$ mixed volume of $P$ and $Q$ and is an important notion of the $L_p$ Brunn-Minkowski theory. The $L_p$ mixed volume is denoted by $V_p(P,Q)$. Clearly we have 
\[V_p(\phi P,Q)=|\det \phi|V_p(P,\phi^{-1}Q)\]
for all $\phi\in\gln$. The above integral representation and the fact that $h(\phi^{-1} Q,u)=h(Q,\phi^{-t}u)$ imply that
\[\int_{\sn}h(Q,u)^p\,dS_p(\phi P,u)=|\det \phi|\int_{\sn}h(Q,\phi^{-t}u)^p\,dS_p(P,u).\]
Note that $p$th powers of support functions are homogeneous of degree $p$. Moreover, Kiderlen \cite{Kid11} showed that differences of $p$th powers of support functions are dense in the space of continuous functions on $\sn$. So the last equation immediately implies that $L_p$ surface area measures are $\gln$ contravariant of degree $p$.

\section{Functional equations}

Let $\lambda\in(0,1)$ and $p\in\R$. We define two families of linear maps on $\R^n$ by
\[
\phi_{\lambda}e_1=e_1, \quad \phi_{\lambda} e_2=(1-\lambda)e_1+\lambda e_2,\quad\phi_{\lambda} e_k=e_k \quad\textrm{for}\,\,3\leq k\leq n,
\]
and
\[
\psi_{\lambda} e_1=(1-\lambda)e_1+\lambda e_2,\quad\psi_{\lambda}e_2=e_2,\quad\psi_{\lambda} e_k=e_k \quad\textrm{for}\,\,3\leq k\leq n.
\]
The following functional equation for $f:\R_+\times \R^n\backslash\{o\}\to\R$ will play a key role:
\begin{equation}\label{fundeq}
f(s,x)=\lambda^{\frac pn}f\left(s\lambda^{\frac 1n},\phi_{\lambda}^tx\right)+(1-\lambda)^{\frac pn}f\left(s(1-\lambda)^{\frac 1n},\psi_{\lambda}^tx\right).
\end{equation}
The next result proves that a function which satisfies (\ref{fundeq}) at certain points is homogeneous in its first argument.

\begin{lemma}\label{homlemma}Let $p\in\R$ and suppose that a function $f:\R_+\!\times\R^n\backslash\{o\}\to\R$ satisfies (\ref{fundeq}). If $x\in\R^n\backslash\{o\}$ is a fixed point of $\phi_{\lambda}^t$ and $\psi_{\lambda}^t$ and $f(\cdot,x)$ is bounded from below on some open interval then
\[f(s,x)=s^{n-p}f(1, x)\]
for every $s>0$.
\end{lemma}

\begin{proof} From (\ref{fundeq}) we see that
\begin{equation}\label{hom1}f(s^{\frac 1n},x)=\lambda^{\frac pn}f(s^{\frac 1n}\lambda^{\frac 1n},x)+(1-\lambda)^{\frac pn}f(s^{\frac 1n}(1-\lambda)^{\frac 1n},x)\end{equation}
for every $s>0$ and $\lambda\in(0,1)$. Define a function $g:\R_+\!\to\R$ by
\[g(s)=f(s^{\frac 1n},x).\]
Then, for every $s>0$ and $\lambda\in(0,1)$, equation (\ref{hom1}) reads as
\begin{equation}\label{hom2}
g(s)=\lambda^{\frac pn}g(s\lambda)+(1-\lambda)^{\frac pn}g(s(1-\lambda)).
\end{equation}
Let $a$ and $b$ be arbitrary positive real numbers. Set
\[s=a+b\qquad\textnormal{and}\qquad\lambda=a(a+b)^{-1}.\] 
If we insert these particular values of $s$ and $\lambda$ in (\ref{hom2}), then we have for all $a,b>0$,
\[(a+b)^{\frac pn}g(a+b)=a^{\frac pn}g(a)+b^{\frac pn}g(b).\]
Thus the function $t\mapsto t^{p/n}g(t)$ solves Cauchy's functional equation on $\R_+$ and, by assumption, it is bounded from below on some open interval. It is well known (see e.g. \cite[Corollary 9]{AczelDhombres}) that this implies that $t^{\frac pn}g(t)=tg(1)$ and hence
\[g(t)=t^{1-\frac pn}g(1).\]
Finally, the definition of $g$ immediately yields
\[f(s,x)=g(s^{n})=s^{n-p}g(1)=s^{n-p}f(1,x).\]
\end{proof}

Next, we are going to show that special solutions of (\ref{fundeq}) are determined by their values on a small set.

\begin{lemma} \label{keq}Let $p\in\R$ and suppose that $f:\R_+\!\times\R^n\backslash\{o\}\to\R$ has the following properties:
\begin{itemize}
\item[(i)] $f$ satisfies (\ref{fundeq}).
\item[(ii)] $f$ is positively homogeneous of degree $-p$ in the second argument.
\item[(iii)] For every $s\in\R^+$ the function $f(s,\cdot)$ has countable support if restricted to $\sn$.
\item[(iv)] For every $x\in\R^n\backslash\{o\}$ the function $f(\cdot,x)$ is bounded from below on some open interval.
\item[(v)] For each $\pi\in\sln$ which is induced by a permutation matrix and all $(s,x)\in \R_+\!\times\R^n\backslash\{o\}$  
\begin{equation}\label{per3}f(s,\pi x)=f(s,x).\end{equation}
\end{itemize}  
If 
\[f(s,x)=0\qquad for\,\, every\qquad  (s,x)\in\R_+\!\times\{\pm e_1\},\]
then 
\[f(s,x)=0\qquad for\,\, every\qquad  (s,x)\in\R^+\!\times\R^n\backslash\lin\{e_1+\cdots+e_n\}.\]
\end{lemma}

\begin{proof} Note that (\ref{fundeq}) gives for all $(s,x)\in\R_+\times\R^n\backslash\{o\}$ and each $\lambda\in(0,1)$,
\begin{eqnarray}
f(s,\psi_{\lambda}^{-t}x)&=&\lambda^{\frac pn}f(s\lambda^{\frac 1n},\phi_{\lambda}^t\psi_{\lambda}^{-t}x)+(1-\lambda)^{\frac pn}f(s(1-\lambda)^{\frac 1n},x),\label{eq71}\\
f(s,\phi_{\lambda}^{-t}x)&=&\lambda^{\frac pn}f(s\lambda^{\frac 1n},x)+(1-\lambda)^{\frac pn}f(s(1-\lambda)^{\frac 1n},\psi_{\lambda}^t\phi_{\lambda}^{-t}x)\label{eq72}.
\end{eqnarray}
For $1\leq j\leq n-1$ we use induction on the number $j$ of non-vanishing coordinates to prove that $f(s,x)=0$ for all $s>0$ and every $x$ with $j$ non-vanishing coordinates. Let $j=1$. Since we can always find a permutation matrix $\pi\in\sln$ with $e_i=\pi e_1$, we have
\[f(s,\pm e_i)=f(s,\pm\pi e_1)=f(s,\pm e_1)=0.\]
Since $f$ is positively homogeneous in the second argument we infer that $f(s,x)=0$ for every $x$ with one non-vanishing coordinate.
Let $1\leq j<n-1$ and suppose that $f(s,x)=0$ for every $s>0$ and every $x\neq o$ with at most $j$ non-vanishing coordinates. By (\ref{per3}) it is enough to prove $f(s,x)=0$ for
 $x=x_1e_1+\cdots+x_{j+1}e_{j+1}$ with $x_1,\ldots,x_{j+1}\neq 0$. Suppose that $0<x_1<x_2$ or $x_2<x_1<0$ and set $\lambda=x_1/x_2$. Then
\begin{eqnarray*}
\psi_{\lambda}^{-t}x&=&x_2e_2+x_3e_3+\cdots+x_{j+1}e_{j+1},\\
\phi_{\lambda}^t\psi_{\lambda}^{-t}x&=&x_1e_2+x_3e_3+\cdots+x_{j+1}e_{j+1}.
\end{eqnarray*}
Relation (\ref{eq71}) and the induction hypothesis show $f(s(1-\lambda)^{\frac 1n},x)=0$
for every $s\in\R_+$, and hence $f(s,x)=0$ for every $s\in\R_+$.  

If $0<x_2<x_1$ or $x_1<x_2<0$, then set $\lambda=(x_1-x_2)/x_1$. Thus
\begin{eqnarray*}
\phi^{-t}_{\lambda}x&=&x_1e_1+x_3e_3+\cdots+x_{j+1}e_{j+1},\\
\psi_{\lambda}^t\phi_{\lambda}^{-t}x&=&x_2e_1+x_3e_3+\cdots+x_{j+1}e_{j+1}.
\end{eqnarray*}
Relation (\ref{eq72}) and the induction hypothesis show $f(s\lambda^{\frac 1n},x)=0$
for every $s\in\R_+$, and hence $f(s,x)=0$ for every $s\in\R_+$.  

If $\textrm{sgn}(x_1)\neq\textrm{sgn}(x_2)$, set $\lambda=x_1/(x_1-x_2)$. Then 
\begin{eqnarray*}
\phi^{t}_{\lambda}x&=&x_1e_1+x_3e_3+\cdots+x_{j+1}e_{j+1},\\
\psi_{\lambda}^tx&=&x_2e_2+x_3e_3+\cdots+x_{j+1}e_{j+1}.
\end{eqnarray*}
It follows directly from (\ref{fundeq}) that $f(s,x)=0$ for every $s\in\R_+$. In conclusion, we proved that for $x=x_1e_1+\cdots+x_{j+1}e_{j+1}$ with $x_1,\ldots,x_{j+1}\neq 0$ and $x_1\neq x_2$ we have $f(s,x)=0$ for every $s\in\R_+$. This and (\ref{per3}) actually show that $f(s,x)=0$ for every $s\in\R_+$ and each $x=x_1e_1+\cdots+x_{j+1}e_{j+1}$ where at least two coordinates are different. It remains to prove that $f(s,x)=0$ for every $s\in\R_+$ and $x=x_1e_1+\cdots+x_1e_{j+1}$ with $x_1\neq0$. By the homogeneity of $f$ in its second argument it suffices to prove that $f(s,x)=0$ for every $s\in\R_+$ and $x=e_1+\cdots+e_{j+1}$ or $x=-e_1-\cdots-e_{j+1}$. We consider only the case $x=e_1+\cdots+e_{j+1}$; the other one is treated similarly. Let $0<\lambda<1$ and set $y=\lambda e_1+e_2+\cdots+e_{j+2}$. Note that
\begin{eqnarray*}
\psi_{\lambda}^{-t}y&=&e_2+e_3+\cdots+e_{j+2},\\
\phi_{\lambda}^t\psi_{\lambda}^{-t}y&=&\lambda e_2+e_3+\cdots+e_{j+2}.
\end{eqnarray*}
By what we have already shown, relation (\ref{per3}), and (\ref{eq71}) we arrive at 
\[f(s,\psi_{\lambda}^{-t}y)=(1-\lambda)^{\frac pn}f(s(1-\lambda)^{\frac 1n}, y).\]
From (\ref{per3}) and Lemma \ref{homlemma} we infer that $f(\cdot,\psi_{\lambda}^{-t}y)$ is positively homogeneous of degree $n-p$, and hence
\[(1-\lambda)^{-1}f(1,\psi_{\lambda}^{-t}y)=f(1, y)\]
for all $0<\lambda<1$. If $f(1,\psi_{\lambda}^{-t}y)=f(1,e_2+e_3+\cdots+e_{j+2})$ were nonzero, then $f(1,y)=f(1,\lambda e_1+e_2+\cdots+e_{j+2})$ would therefore be nonzero for all $0<\lambda<1$. But since $f$ is positively homogeneous in the second argument, this would contradict the assumption that $f(1,\cdot)$ has countable support on $\sn$. By homogeneity, for each $s\in\R^+$, we have $f(s,\psi_{\lambda}^{-t}y)=0$ and hence (\ref{per3}) gives $f(s,e_1+\cdots+e_{j+1})=0$. This concludes the induction. 

We showed that $f(s,x)=0$ for $s\in\R_+$ and points $x$ with at most $n-1$ non-vanishing coordinates. As in the first part of the induction we see that also $f(s,x)=0$ for $s\in\R_+$ and points $x$ with $n$ non-vanishing coordinates provided that at least two of them are different.
\end{proof}

\section{The case $p=1$}

\begin{lemma}\label{lower1} Suppose that $\mu:\cP_o^n\to\cM^d(\sn)$ is $\sln$ contravariant of degree $1$. Then there exists a constant $a\in\R_+\!$ such that
\[\mu(sT',\cdot)=as^{n-1}(\delta_{e_1}+\delta_{-e_1})\]
for every $s\in\R_+$ and $\mu(\{o\},\cdot)=0$, where $T'$ denotes the $(n-1)$-dimensional standard simplex in $e_1^\bot$.
\end{lemma}

\begin{proof} First, we show that $\mu(P,\cdot)$ is supported at $\pm e_k$ provided that $P\subset e_k^{\bot}$. So let $P\subset e_k^{\bot}$ and set $f(P,x)=\mu(P,\langle x\rangle)|x|^{-1}$ for $x\in\R^n\backslash\{o\}$. Suppose that $x\in\R^n$ is given with $x_j\neq 0$ for some $j\neq k$. For $t\in\R$ define $\phi\in\sln$ by
\[\phi e_k=e_k+te_j,\quad\phi e_i=e_i,\quad i\in\{1,\ldots,n\}\backslash\{k\}.\]
Since $\phi P=P$ we get from (\ref{eqn1}) that
\[f(P,x)=f(\phi P,x)=f(P,\phi^t x)=f(P,(x_1,\ldots, x_{k-1},x_k+tx_j,x_{k+1},\ldots,x_n)).
\]
If $f(P,x)\neq 0$, then 
\[f(P,(x_1,\ldots, x_{k-1},x_k+tx_j,x_{k+1},\ldots,x_n))\neq 0\] for all $t\in\R$.
Since these correspond to uncountably many points on $\sn$ and $f(P,\cdot)$ restricted to $\sn$ has countable support, $f(P,x)$ has to be zero. Consequently, $f(P,\cdot)$ is supported only at $\pm e_k$.

 In particular, the measure $\mu(\{o\},\cdot)$ has to be supported at $\pm e_1$ as well as $\pm e_2$. Therefore, it has to be zero. It remains to prove the formula for $\mu(sT',\cdot)$. We already know that $f(sT',\cdot)$ is supported only at $\pm e_1$. Thus
\[f(sT',x)=a_1(s)\delta_{e_1}(\langle x\rangle)|x|^{-1}+a_2(s)\delta_{-e_1}(\langle x\rangle)|x|^{-1}.\]
Define $\psi\in\sln$ by
\[\psi e_1=-e_1,\quad \psi e_2=e_3, \quad \psi e_3=e_2,\quad \psi e_i=e_i,\quad 4\leq i\leq n.\]
Then by (\ref{eqn1}) and the relation $sT'=\psi(sT')$ we have $f(sT',e_1)= f(sT',-e_1)$ and therefore $a_1(s)=a_2(s)$. Finally, define $\tau\in\sln$ by
\[\tau e_1=s^{1-n}e_1,\quad \tau e_i=se_i,\quad 2\leq i\leq n.\]
Then $sT'=\tau T'$ and relation (\ref{eqn1}) again show that
\[f(sT',e_1)=s^{n-1}f(T',e_1),\]
which proves $a_1(s)=s^{n-1}a(1)$. Now set $a=a(1)$. \end{proof}

Now, we are in a position to prove Theorem \ref{main1}. For the reader's convenience we will repeat its statement.
\begin{theorem} A map $\mu:\cP_o^n\to\cM(\sn)$ is an $\sln$ contravariant valuation of degree 1 if and only if there exist constants $c_1,c_2,c_3,c_4\in\R$ with  $c_1,c_2\geq0$ and $c_1+c_3\geq0$, $c_2+c_4\geq 0$ such that
\begin{equation}\label{eq111}\mu(P,\cdot)=c_1S(P,\cdot)+c_2S(-P,\cdot)+c_3S^*(P,\cdot)+c_4S^*(-P,\cdot)\end{equation}
for every $P\in\cP_o^n$.
\end{theorem}

\begin{proof} Set $f(P,x)=\mu^d(P,\langle x\rangle)|x|^{-1}$, $x\in\R^n\backslash\{o\}$, and define constants
\[d_1=(n-1)!(f(T',e_1)-f(T^n,e_1)),\qquad d_2=(n-1)!f(T^n,e_1).\]
By Lemma \ref{lower1} we have
$$ f(sT',x) = f(T',e_1)s^{n-1} \left( \delta_{e_1}(\langle x\rangle) + \delta_{-e_1}(\langle x\rangle) \right) |x|^{-1} .$$
Hence, the function
\[g(P,x)=f(P,x)-d_1S^o(P,x)-d_2S^o(-P,x)\]
is a valuation which vanishes on $sT'$. For $\lambda\in (0,1)$ let $H_{\lambda}$ be the hyperplane through $o$ with normal vector $\lambda e_1-(1-\lambda)e_2$. Note that
\[(sT^n)\,\cap\, H_{\lambda}^+=s\phi_{\lambda} T^n,\quad (sT^n)\,\cap\, H_{\lambda}^-=s\psi_{\lambda} T^n\quad\textnormal{and}\quad (sT^n)\,\cap\, H_{\lambda}=s\phi_{\lambda} T'.\]
So the valuation property of $g$ implies that for all $(s,x)\in\R_+\times\R^n\backslash\{o\}$,
\[
g(sT^n,x) + g(s\phi_{\lambda} T', x)= g(s\phi_{\lambda}T^n, x  )+ g(s \psi_{\lambda}T^n,x).
\]
By (\ref{eqn1}) we have $g(s\phi_{\lambda} T',x)=\lambda^{1/n}g(s\lambda^{1/n}T',\phi_{\lambda}^tx)$, $g(s\phi_{\lambda} T^n,x)=\lambda^{1/n}g(s\lambda^{1/n}T^n,\phi_{\lambda}^tx)$, and $g(s\psi_{\lambda} T^n,x)=(1-\lambda)^{1/n}g(s(1-\lambda)^{1/n}T^n,\psi_{\lambda}^tx)$. This and the fact that $g$ vanishes on multiples of $T'$ proves
\begin{equation}\label{556}
g(sT^n,x)= \lambda^{\frac1n} g(s\lambda^{\frac1n}T^n,\phi_{\lambda}^t x  )+ (1-\lambda)^{\frac1n}g(s(1-\lambda)^{\frac1n} T^n,\psi_{\lambda}^tx).
\end{equation}
Consequently, the map $(s,x)\mapsto g(sT^n,x)$ satisfies (\ref{fundeq}). By Lemma \ref{homlemma} we know that $g(sT^n,e_3)=s^{n-1}g(T^n,e_3)$ for all positive $s$. But $g(sT^n,e_1)=g(sT^n,e_3)$ and $g(T^n,e_1)=0$ by definition. Thus $g(sT^n,e_1)=0$ for all $s\in\R_+$.

For $x=x_1e_1+x_2e_2$ with $x_1>0$ and $x_2<0$ set $\lambda=\frac{x_1}{x_1-x_2}$. Evaluating (\ref{556}) for this $x$ and $\lambda$, Lemma \ref{homlemma} together with the homogeneity of $g$ of degree $-1$ in the second argument, the equality $g(sT^n,e_1)=0$ and the definition of $\lambda$ show that
\begin{eqnarray*}
g(sT^n,x)&=&\lambda^{\frac 1n}g(s\lambda^{\frac 1n}T^n,x_1e_1)+(1-\lambda)^{\frac 1n}g(s(1-\lambda)^{\frac 1n}T^n,x_2e_2)\\
        &=&\frac{\lambda}{x_1}g(sT^n,e_1)+\frac{1-\lambda}{-x_2}g(sT^n,-e_2)\\
        &=&\frac{1}{x_1-x_2}g(sT^n,-e_2).
\end{eqnarray*}
Since $g(sT^n,-e_2)=g(sT^n,-e_1)$ and $g(sT^n,\cdot)$ has at most countable support if restricted to $\sn$, we obtain from the last lines that also $g(sT^n,-e_1)=0$. From Lemma \ref{keq} we further deduce that $g(sT^n,\cdot)$ is supported only at $\pm\langle e_1+\cdots+e_n\rangle$. Define constants
\[d_3=\frac{g(T^n,\langle e_1+\cdots+e_n\rangle)}{S^*(T^n,\langle e_1+\cdots+e_n\rangle)},\qquad d_4 =\frac{g(T^n,-\langle e_1+\cdots+e_n\rangle)}{S^*(-T^n,-\langle e_1+\cdots+e_n\rangle)}.\]
Thus, by Lemma \ref{homlemma}, we have $g(sT^n,x)= d_3S^*(sT^n,x)+d_4S^*(-sT^n,x)$
and consequently
\[
\mu^d(sT^n,x)=d_1S(sT^n,x)+d_2S(-sT^n,x)+(d_3-d_1)S^*(sT^n,x)+(d_4-d_2)S^*(-sT^n,x).
\]
By Lemma \ref{ext} this proves (\ref{eq111}) for the discrete part $\mu^d$. 

Note that $P\mapsto\C_1\mu(P,\cdot)$ is a function from $\cP_o^n$ to $C_1^+(\mathbb{R}^n)$ which is an even $\sln$ contravariant valuation. The linearity of the cosine transform gives $\C_1\mu(P,\cdot)=\C_1\mu^c(P,\cdot)+\C_1\mu^d(P,\cdot)$. From Theorem \ref{funcchar1} we know that
\[\C_1\mu(P,\cdot)=\C_1 \left( d_5S(P,\cdot)+d_6S^*(P,\cdot) \right) .\]
By the discrete case we just established we know that
\[\C_1\mu^d(P,\cdot)=\C_1 \left( d_7S(P,\cdot)+d_8S^*(P,\cdot) \right) ,\]
and hence also
\[\C_1\mu^c(P,\cdot)=\C_1 \left( d_9S(P,\cdot)+d_{10}S^*(P,\cdot) \right) .\]
In particular, we have $\C_1\mu^c(sT',\cdot)=d_9\C_1S(sT',\cdot)$. By the injectivity property (\ref{inject}) we know that $\mu^c(sT',e_1)+\mu^c(sT',-e_1)=2s^{n-1}d_9/(n-1)!$. But since $\mu^c(sT',\cdot)$ is continuous we have $d_9=0$. Thus
\[\C_1\mu^c(sT^n,\cdot)(e_i-e_j)=d_{10}\C_1S^*(sT^n,\cdot)(e_i-e_j)=0\] 
for $1\leq i\neq j\leq n$. So $\mu^c(sT^n,\cdot)$ has to be concentrated on each hyperplane $\{x_i=x_j\}$, and consequently it is concentrated at the two points $\pm\langle e_1+\cdots+e_n\rangle$. The continuity of $\mu^c(sT^n,\cdot)$ therefore implies $\mu^c(sT^n,\cdot)=0$. So by the discrete case we have
\[\mu(sT^n,\cdot)=\mu^d(sT^n,\cdot)=c_1S(sT^n,\cdot)+c_2S(-sT^n,\cdot)+c_3S^*(sT^n,\cdot)+c_4S^*(-sT^n,\cdot).\]
for some constants $c_1,\ldots,c_4\in\R$. 

Next, we want to prove that $\mu(\{o\},\cdot)=0$. For $s>0$ define $\phi\in\sln$ by
\[\phi e_1 =s^{-1} e_1,\quad \phi e_2 =s e_2,\quad \phi e_k= e_k,\quad 3\leq k\leq n.\]
Since $\phi\{o\}=\{o\}$ the $\sln$ contravariance of $\mu$ implies
\begin{eqnarray*}
\mu(\{o\},\sn)&=&\int_{\sn}|x|\,d\mu(\phi\{o\},x)\\
&=&\int_{\sn}(s^2x_1^2+s^{-2}x_2^2+x_3^2+\cdots+x_n^2)^{\frac12}\,d\mu(\{o\},x).
\end{eqnarray*}
Take the limit $s\to\infty$ in the above equation. Then Fatou's lemma implies 
\[\mu(\{o\},\sn)\geq \infty\cdot\mu(\{o\},\{x_1\neq 0\}).\]
Since $\mu(\{o\},\cdot)$ is finite, $\mu(\{o\},\cdot)$ is supported at $\{x_1=0\}$. Similarly, one shows that $\mu(\{o\},\cdot)$ is supported at $\{x_j=0\}$ for $j=2,\ldots,n$. This immediately implies that $\mu(\{o\},\cdot)=0$.

Lemma \ref{ext} and the fact that $S(P,\cdot)=S^o(P,\cdot)+S^*(P,\cdot)$ therefore prove 
\begin{eqnarray*}
\mu(P,\cdot)&=&c_1S(P,\cdot)+c_2S(-P,\cdot)+c_3S^*(P,\cdot)+c_4S^*(-P,\cdot)\\
&=&c_1S^o(P,\cdot)+c_2S^o(-P,\cdot)+(c_1+c_3)S^*(P,\cdot)+(c_2+c_4)S^*(-P,\cdot)
\end{eqnarray*}
for every $P\in\cP_o^n$. In particular
\[
0\leq\mu(T^n, e_1)=c_2/(n-1)!\qquad\textnormal{and}\qquad 0\leq\mu(T^n,-e_1)=c_1/(n-1)!,
\]
as well as
\begin{eqnarray*}
0&\leq&\mu(T^n,\langle e_1+\cdots+e_n\rangle)=\sqrt{n}(c_1+c_3)/(n-1)!,\\
0&\leq&\mu(T^n,-\langle e_1+\cdots+e_n\rangle)=\sqrt{n}(c_2+c_4)/(n-1)!,
\end{eqnarray*}
which proves the asserted relations for the constants.
\end{proof}

\section{The case $p\neq 1$}

We begin by studying $\sln$ contravariant valuations on lower dimensional polytopes. 

\begin{lemma} \label{lemma1}Suppose that $p\neq 1$ and $\mu:\cP_o^n\to\cM(\sn)$ is an $\sln$ contravariant valuation of degree $p$ which vanishes on polytopes of dimension less than $n-1$. Let $s\in\R_+$. If $\mu(sT',\cdot)$ is supported only at the two points $\pm e_1$, then $\mu(sT',\cdot)=0$.
\end{lemma}

\begin{proof}
For $\lambda\in(0,1)$ let $H_{\lambda}$ denote the hyperplane containing the origin with normal vector $\lambda e_2- (1-\lambda)e_3$. The valuation property of $\mu$ yields
\begin{equation}\label{val1}
\mu(sT',\cdot)+\mu(sT'\cap H_{\lambda},\cdot)=\mu(sT'\cap H_{\lambda}^+,\cdot)+\mu(sT'\cap H_{\lambda}^-,\cdot).
\end{equation}
Define maps $\sigma,\tau \in\sln$ by
\[\sigma e_1=\frac 1\lambda e_1,\quad\sigma e_2=e_2,\quad\sigma e_3=(1-\lambda)e_2+\lambda e_3,\quad\sigma e_k=e_k,\quad 4\leq k\leq n,\]
and
\[\tau e_1=\frac {1}{1-\lambda} e_1,\quad\tau e_2=(1-\lambda)e_2+\lambda e_3,\quad\tau e_k=e_k,\quad 3\leq k\leq n.\]
Since
\[sT'\cap H_{\lambda}^+=\sigma (sT')\quad\textnormal{and}\quad sT'\cap H_{\lambda}^-=\tau (sT')\]
and since $\mu$ vanishes on $sT'\cap H_{\lambda}$, we have
\[\mu(sT',\cdot)=\mu(\sigma (sT'),\cdot)+\mu(\tau (sT'),\cdot).\]
Thus by (\ref{eqn1}) we obtain
\begin{eqnarray*}
\mu(sT',\pm e_1)&=&\mu(\sigma (sT'),\pm e_1)+\mu(\tau (sT'),\pm e_1)\\
&=&\mu(sT',\pm \langle\sigma^t e_1\rangle)|\sigma^t e_1|^{-p}+\mu(sT',\pm\langle\tau^t e_1\rangle)|\tau^t e_1|^{-p}\\
&=&(\lambda^p +(1-\lambda)^p)\mu(sT',\pm e_1).
\end{eqnarray*}
Since $p\neq 1$, we get $\mu(sT',\pm e_1)=0$.
\end{proof}

\subsection{The discrete case}

The next result concerns simplicity for valuations with discrete images.

\begin{lemma}\label{simpl2} Let $p\neq 1$ and suppose that $\mu:\cP_o^n\to\cM^d(\sn)$ is $\sln$ contravariant of degree $p$. Then
$\mu(sT',\cdot)=0$
for every $s\in\R^+$ and $\mu(\{o\},\cdot)=0$.
\end{lemma}

\begin{proof}  By Lemma \ref{lemma1} it is enough to show that $\mu(P,\cdot)$ is supported at $\pm e_k$ provided that $P\subset e_k^{\bot}$. So let $P\subset e_k^{\bot}$ and set $f(P,x)=\mu(P,\langle x\rangle)|x|^{-p}$ for $x\in\R^n\backslash\{o\}$. Suppose that $x\in\R^n$ is given with $x_j\neq 0$ for some $j\neq k$. For $t\in\R$ define $\phi\in\sln$ by
\[\phi e_k=e_k+te_j,\quad\phi e_i=e_i,\quad i\in\{1,\ldots,n\}\backslash\{k\}.\]
Since $\phi P=P$, we get from (\ref{eqn1}) that
\[f(P,x)=f(\phi P,x)=f(P,\phi^t x)=f(P,(x_1,\ldots, x_{k-1},x_k+tx_j,x_{k+1},\ldots,x_n)).
\]
If $f(P,x)\neq 0$, then 
\[f(P,(x_1,\ldots, x_{k-1},x_k+tx_j,x_{k+1},\ldots,x_n))\neq 0\] for all $t\in\R$.
Since these correspond to uncountably many points on $\sn$ and $f(P,\cdot)$ restricted to $\sn$ has countable support, $f(P,x)$ has to be zero. Consequently, $f(P,\cdot)$ is supported only at $\pm e_k$. \end{proof}

The main result of this section is the following classification of $\sln$ contravariant valuations of degree $p$ with discrete images.

\begin{theorem}\label{main3} Let $p\in\R\backslash\{1\}$. A map $\mu:\cP_o^n\to\cM^d(\sn)$ is an $\sln$ contravariant valuation of degree $p$ if and only if there exist nonnegative constants $c_1,c_2\in\R$ such that
\[\mu(P,\cdot)=c_1S_p(P,\cdot)+c_2S_p(-P,\cdot)\]
for every $P\in\cP_o^n$.
\end{theorem}

By Lemmas \ref{ext} and \ref{simpl2}, this theorem will be an immediate consequence of the following 

\begin{lemma}\label{l1} Let $p\neq 1$ and $\mu:\cP_o^n\to\cM(\sn)$ be an $\sln$ contravariant valuation of degree $p$. If, for all $s>0$, $\mu(sT',\cdot)=0$ and $\mu(sT^n,\cdot)$ is discrete, then there exist constants $c_1,c_2\geq0$ such that for all $s>0$
\[\mu(sT^n,\cdot)=c_1S_p(sT^n,\cdot)+c_2S_p(-sT^n,\cdot).\]
\end{lemma}

\begin{proof}For $\lambda\in (0,1)$ let $H_{\lambda}$ be the hyperplane through $o$ with normal vector $\lambda e_1-(1-\lambda)e_2$. Note that
\[(sT^n)\,\cap\, H_{\lambda}^+=s\phi_{\lambda} T^n,\quad (sT^n)\,\cap\, H_{\lambda}^-=s\psi_{\lambda} T^n,\quad\textnormal{and}\quad (sT^n)\,\cap\, H_{\lambda}=s\phi_{\lambda} T'.\] Set $f(P,x)=\mu(P,\langle x\rangle)|x|^{-p}$. Then the valuation property of $\mu$ implies that for all $(s,x)\in\R_+\times\R^n\backslash\{o\}$,
\[
f(sT^n,x) + f(s\phi_{\lambda} T', x)= f(s\phi_{\lambda}T^n, x  )+ f(s \psi_{\lambda}T^n,x).
\]
From (\ref{eqn1}) and the assumption that $\mu$ vanishes on multiples of $T'$ we deduce that
\[f(sT^n,x)= \lambda^{\frac pn} f(s\lambda^{\frac1n}T^n,\phi_{\lambda}^t x  )+ (1-\lambda)^{\frac pn}f(s(1-\lambda)^{\frac1n} T^n,\psi_{\lambda}^tx).
\]
Consequently, the map $(s,x)\mapsto f(sT^n,x)$ satisfies (\ref{fundeq}). Evaluating this equality at the first canonical basis vector $e_1$ proves
\begin{equation}\label{eq3}
f(s,e_1)=\lambda^{\frac pn}f(s\lambda^{\frac 1n},e_1+(1-\lambda)e_2)+(1-\lambda)^{\frac pn}f(s(1-\lambda)^{\frac 1n},(1-\lambda)e_1).
\end{equation}
But $e_3$ is a fixpoint of $\phi_{\lambda}^t$ and $\psi_{\lambda}^t$ and $f(s,e_3)=f(s,e_1)$ for all positive $s$, hence by Lemma \ref{homlemma} we get 
\[f(s,e_1)=s^{n-p}f(1,e_1) \quad\textnormal{for all } s>0.\]
Thus we obtain from (\ref{eq3}) that
\[(1-(1-\lambda)^{1-p})f(s,e_1)=\lambda^{\frac pn}f(s\lambda^{\frac 1n},e_1+(1-\lambda)e_2).\]
This shows in particular that 
\[f(s\lambda^{\frac 1n},e_1+(1-\lambda)e_2)=s^{n-p}f(\lambda^{\frac 1n},e_1+(1-\lambda)e_2)\]
for all $s>0$ and each $\lambda\in (0,1)$. Consequently
\[f(s\lambda^{\frac 1n},e_1+(1-\lambda)e_2)=s^{n-p}\lambda^{1-\frac pn}f(1,e_1+(1-\lambda)e_2)\]
and hence
\[(1-(1-\lambda)^{1-p})\lambda^{-1} f(s,e_1)=f(s,e_1+(1-\lambda)e_2).\]
Since $p\neq 1$ and for fixed $s$ the function $f(s,\cdot)$ has countable support if restricted to $\sn$, we see that $f(s,e_1)=0$ for all $s>0$. Similarly, by looking at $-e_1$ instead of $e_1$ in the above argument, we infer $f(s,-e_1)=0$ for all $s>0$. From Lemma \ref{keq} we deduce that $f(s,\cdot)$ is supported only at $\pm \langle e_1+\cdots+e_n\rangle$. Define constants
\[c_1=\frac{f(1,\langle e_1+\cdots+e_n\rangle)}{S_p(T^n,\langle e_1+\cdots+e_n\rangle)}\quad\textnormal{and}\quad c_2=\frac{f(1,-\langle e_1+\cdots+e_n\rangle)}{S_p(-T^n,-\langle e_1+\cdots+e_n\rangle)}.\]
Since $\pm \langle e_1+\cdots+e_n\rangle$ are fixpoints of $\phi_{\lambda}^t$ and $\psi_{\lambda}^t$, we therefore have by Lemma \ref{homlemma} that
\[f(s,x)=c_1S_p(sT^n,x)+c_2S_p(-sT^n,x)\]
for all $s>0$ and all $x\in\sn$. Now, relation (\ref{eqivmeas}) concludes the proof.
\end{proof}

\subsection{The case $p>0$}

\begin{lemma}\label{simpllem1} Let $p\in\R_+\backslash\{1\}$. If $\mu:\cP_o^n\to\cM(\sn)$ is an $\sln$ contravariant valuation of degree $p$, then $\mu(sT',\cdot)=0$ for every $s>0$ and $\mu(\{o\},\cdot)=0$.
\end{lemma}

\begin{proof}By Lemma \ref{lemma1} it is enough to show that $\mu(P,\cdot)$ is supported at $\pm e_k$ provided that $P\subset e_k^{\bot}$. So let $P\subset e_k^{\bot}$. For $t\in\R$ and $j\neq k$ define $\phi\in\sln$ by
\[\phi e_k=e_k+te_j,\quad\phi e_i=e_i,\quad i\in\{1,\ldots,n\}\backslash\{k\}.\]
Note that $\phi P=P$. Therefore the $\sln$ contravariance of $\mu$ implies
\begin{eqnarray*}
\mu(P,\sn)&=&\mu(\phi P,\sn)=\int_{\sn}|x|^p\,d\mu(\phi P,x)\\
&=&\int_{\sn}\left(x_1^2+\cdots+x_{k-1}^2+(x_k-tx_j)^2+x_{k+1}^2+\cdots+x_n^2\right)^{\frac p2}d\mu(P,x)
\end{eqnarray*}
If $\mu(P,\{x_j\neq0\})\neq0$, then Fatou's lemma shows that the last integral goes to infinity as $t$ tends to infinity. But this would contradict the finiteness of $\mu(P,\cdot)$. Thus $\mu(P,\{x_j\neq0\})=0$, which immediately implies that $\mu(P,\cdot)$ is supported only at $\pm e_k$. 
\end{proof}

Now, we establish the characterization of $L_p$ surface area measures for positive $p$ which was already announced in the introduction. We emphasize again that for positive $p$ no homogeneity assumptions are needed. It suffices to assume $\sln$ contravariance instead of $\gln$ contravariance.

\begin{theorem}\label{main4} Let $p\in\R_+\backslash\{1\}$. A map $\mu:\cP_o^n\to\cM(\sn)$ is an $\sln$ contravariant valuation of degree $p$ if and only if there exist nonnegative constants $c_1,c_2\in\R$ such that
\[\mu(P,\cdot)=c_1S_p(P,\cdot)+c_2S_p(-P,\cdot)\]
for every $P\in\cP_o^n$.
\end{theorem}

\begin{proof}
Note that $\C_p\mu(P,\cdot):\cP_o^n\to C_p^+(\R^n)$ is an even $\sln$ contravariant valuation. Thus, by Theorem \ref{funcchar2}, we have
\[\C_p\mu(sT^n,\cdot)(e_i-e_j)=c\,\C_p S_p(sT^n,\cdot)(e_i-e_j)=0.\]
for $1\leq i\neq j\leq n$. So $\mu(sT^n,\cdot)$ has to be concentrated on each hyperplane $\{x_i=x_j\}$, and consequently it is concentrated at the two points $\pm\langle e_1+\cdots+e_n\rangle$. By Lemmas \ref{l1} and \ref{simpllem1} there exist nonnegative constants $c_1$ and $c_2$ such that
\[\mu(sT^n,x)=c_1S_p(sT^n,x)+c_2S_p(-sT^n,x)\]
for all $s>0$ and $x\in\sn$. Thus, Lemma \ref{ext} concludes the proof.
\end{proof}

\subsection{The case $p\leq 0$}

Next, we prove the $\gln$ contravariant case for non-positive $p$. We start with the simplicity in this case.
\begin{lemma}\label{simpl22} Let $p\leq 0$ and $\mu:\cP_o^n\to\cM(\sn)$ be a $\gln$ contravariant map of degree $p$. Then $\mu$ is simple, i.e. it vanishes on polytopes of dimension less than $n$. 
\end{lemma}

\begin{proof} Without loss of generality we assume that $P\subset e_1^{\bot}$. For $s>0$ define $\phi\in\gln$ by
\[\phi e_1=se_1,\qquad\phi e_k=e_k,\qquad 2\leq k\leq n.\]
Since $\phi P=P$, the $\gln$ contravariance of $\mu$ yields
\begin{eqnarray*}
\mu(P,\sn)&=&\mu(\phi P,\sn)=\int_{\sn}|x|^p\,d\mu(\phi P,x)\\
&=&s\int_{\sn}\left(\frac{x_1^2}{s^2}+x_2^2+\cdots+x_n^2\right)^{\frac p2}\mu(P,x).
\end{eqnarray*}
Note that the integrand of the last integral is equal to $\left( 1+x_1^2(1/s^2-1) \right)^{p/2}$. Clearly, this function is greater than or equal to $(1+1/s^2)^{p/2}$ and hence
\[\mu(P,\sn)\geq \mu(P,\sn)s\left(1+\frac{1}{s^2}\right)^{\frac p2}.\]
Now take the limit $s\to\infty$ and use the fact that $\mu(P,\cdot)$ is a finite measure in order to arrive at $\mu(P,\sn)=0$. Hence $\mu(P,\cdot)=0$.
\end{proof}

We are now in a position to prove Theorem \ref{main5} for non-positive $p$.

\begin{lemma}\label{main6} Let $p\leq0$. A map $\mu:\cP_o^n\to\cM(\sn)$ is a $\gln$ contravariant valuation of degree $p$ if and only if there exist nonnegative constants $c_1,c_2\in\R$ such that
\[\mu(P,\cdot)=c_1S_p(P,\cdot)+c_2S_p(-P,\cdot)\]
for every $P\in\cP_o^n$.
\end{lemma}

\begin{proof} For $\lambda\in (0,1)$ let $H_{\lambda}$ be the hyperplane through $o$ with normal vector $\lambda e_1-(1-\lambda)e_2$. Recall that
\[T^n\,\cap\, H_{\lambda}^+=\phi_{\lambda} T^n,\quad T^n\,\cap\, H_{\lambda}^-=\psi_{\lambda} T^n\quad\textnormal{and}\quad T^n\,\cap\, H_{\lambda}=\phi_{\lambda} T'.\]
The valuation property, the simplicity derived in Lemma \ref{simpl22}, and the $\gln$ contravariance of $\mu$ yield
\[
\int_{\sn}f\, d\mu(T^n,\cdot)=\lambda\int_{\sn}f\circ\phi_{\lambda}^{-t}\, d\mu(T^n,\cdot)+(1-\lambda)\int_{\sn}f\circ\psi_{\lambda}^{-t}\, d\mu(T^n,\cdot).
\]
Thus, for every continuous $p$-homogeneous $f:\R^n\backslash\{o\}\to\R$, we have
\begin{equation}\label{fundeq3}
0=\int_{\sn}f\circ\phi_{\lambda}^{-t}\,d\mu(T^n,\cdot)+\int_{\sn}\frac{(1-\lambda)(f\circ\psi_{\lambda}^{-t})-f}{\lambda}\,d\mu(T^n,\cdot).
\end{equation}
First, assume $p<0$. Define $f_p(x)=\left((x_1-x_2)^2+x_2^2+\cdots+x_n^2\right)^{p/2}$ for $x\in\R^n\backslash\{o\}$. Note that $f_p$ is strictly positive and $p$-homogeneous. Set
\[g_p(\lambda,x)=f_p\circ\psi_{\lambda}^{-t}=\left(\left(\frac{x_1-x_2}{1-\lambda}\right)^2+x_2^2+\cdots+x_n^2\right)^{\frac p2}.\]
An elementary calculation shows that
\begin{equation}\label{eq77}\frac{\partial}{\partial \lambda}(1-\lambda)g_p(\lambda, x)=-g_p(\lambda,x)+p\,g_p(\lambda,x)^{1-\frac2p}\frac{(x_1-x_2)^2}{(1-\lambda)^2},
\end{equation}
and hence
\begin{eqnarray}
\lim_{\lambda\to 0^+}\frac{(1-\lambda)(f_p\circ\psi_{\lambda}^{-t})-f_p}{\lambda}&=&\left.\frac{\partial }{\partial\lambda}(1-\lambda)g_p(\lambda,x)\right|_{\lambda=0}\label{eq100}\\
&=&-f_p(x)+pf_p(x)^{1-\frac 2p}(x_1-x_2)^2.\nonumber
\end{eqnarray}
Clearly we have
\begin{eqnarray*}
\lim_{\lambda\to 0^+}(f_p\circ\phi_{\lambda}^{-t})(x)&=&\lim_{\lambda\to 0^+}\left(\left(\frac{x_1-x_2}{\lambda}\right)^2+\left(\frac{x_2-x_1}{\lambda}+x_1\right)^2+x_3^2+\cdots+x_n^2\right)^{\frac p2}\\
&=&\mathbb{I}_{\{x_1=x_2\}}(x)f_p(x).
\end{eqnarray*}
If we take the limit $\lambda\to 0^+$ in (\ref{fundeq3}) and are allowed to interchange limit and integrals, then
\[0=\int_{\sn}f_p(x)\left(\mathbb{I}_{\{x_1=x_2\}}(x)-1+pf_p(x)^{-\frac 2p}(x_1-x_2)^2\right)\,d\mu(T^n,x).\]
Since the last integrand is less than or equal to zero and equal to zero precisely for points $x$ with $x_1=x_2$, we deduce that $\mu(T^n,\cdot)$ is concentrated on $\{x_1=x_2\}$. Since by the $\gln$ contravariance $\mu(T^n,\cdot)$ is invariant with respect to coordinate changes we get that $\mu(T^n,\cdot)$ is concentrated at the points $\pm\langle e_1+\cdots+e_n\rangle$. In particular, $\mu(T^n,\cdot)$ is discrete. 

We still have to prove that we are actually allowed to interchange limit and integrals in the above argument. Since $g_p(\lambda,x)$ is obviously  bounded on $(0,1/2]\times\sn$, equations (\ref{eq77}) and (\ref{eq100}), the mean value theorem, and the dominated convergence theorem show that we can interchange limit and integration in the second integral. In order to apply the dominated convergence theorem for the first integral it is enough to prove that the function $f_p\circ\phi_{\lambda}^{-t}$ is bounded from above on $(0,1)\times\sn$. Recall that
\[f_p\circ\phi_{\lambda}^{-t}(x)=\left(\left(\frac{x_1-x_2}{\lambda}\right)^2+\left(x_1+\frac{x_2-x_1}{\lambda}\right)^2+x_3^2+\cdots+x_n^2\right)^{\frac p2}.\]
Define a set
\[U=\{x\in\sn:\,\,(x_1-x_2)^2\geq 1/25\,\,\textnormal{ or }\,\, x_3^2+\cdots+x_n^2\geq 1/25\}.\]
By the negativity of $p$, the function $f_p\circ\phi_{\lambda}^{-t}$ is obviously bounded on $(0,1)\times U$. So suppose that $x\in \sn\backslash U$, i.e. $(x_1-x_2)^2< 1/25$ and $x_3^2+\cdots+x_n^2< 1/25$. From these two inequalities and the fact that $(x_1-x_2)^2+2x_1x_2+x_3^2+\cdots+x_n^2=1$ we get $2x_1x_2>1-2/25=23/25$. Since $|x_2|\leq 1$ we further obtain
\begin{equation}\label{eq97}
|x_1|>\frac{23}{50}.
\end{equation}
Assume that there exists a $\lambda\in(0,1)$ such that
\[\left(\frac{x_1-x_2}{\lambda}\right)^2+\left(x_1+\frac{x_2-x_1}{\lambda}\right)^2+x_3^2+\cdots+x_n^2<\frac{1}{25}.\]
Then we obviously must have
\begin{equation}\label{eq87}\left(\frac{x_1-x_2}{\lambda}\right)^2<\frac{1}{25}\end{equation}
as well as
\begin{equation}\label{eq98}
\left(x_1+\frac{x_2-x_1}{\lambda}\right)^2<\frac{1}{25}.
\end{equation}
But by the reverse triangle inequality, (\ref{eq97}), and (\ref{eq87}) we obtain
\[\left|x_1+\frac{x_2-x_1}{\lambda}\right|\geq|x_1|-\left|\frac{x_2-x_1}{\lambda}\right|>\frac 15,\]
which contradicts (\ref{eq98}). Consequently, $f_p\circ\phi_{\lambda}^{-t}$ is bounded also on $(0,1)\times\sn\backslash U$, which immediately gives the desired result.

For the case $p=0$ set $h(x)=|x_3|f_{-1}(x)$, $x\in\R^n\backslash\{o\}$. Note that $h$ is a nonnegative continuous $0$-homogeneous function. Thus
\begin{eqnarray*}
0&=&\int_{\sn}h\circ\phi_{\lambda}^{-t}\,d\mu(T^n,\cdot)+\int_{\sn}\frac{(1-\lambda)(h\circ\psi_{\lambda}^{-t})-h}{\lambda}\,d\mu(T^n,\cdot)\\
&=&\int_{\sn}|x_3|(f_{-1}\circ\phi_{\lambda}^{-t})\,d\mu(T^n,\cdot)+\int_{\sn}|x_3|\frac{(1-\lambda)(f_{-1}\circ\psi_{\lambda}^{-t})-f_{-1}}{\lambda}\,d\mu(T^n,\cdot)\end{eqnarray*}
As before, we can take the limit $\lambda\to 0^+$ in this equation and interchange limit and integration in order to arrive at
\[0=\int_{\sn}|x_3|f_{-1}(x)\left(\mathbb{I}_{\{x_1=x_2\}}(x)-1-f_{-1}(x)^{2}(x_1-x_2)^2\right)d\mu(T^n,\cdot).\]
Thus $\mu(T^n,\cdot)$ has to be concentrated on $\{x_3=0\}\cup\{x_1=x_2\}$. Since $\mu(T^n,\cdot)$ is invariant under coordinate changes, $\mu(T^n,\cdot)$ has to be supported at 
\[\bigcap\big\{\{x_i=0\}\cup\{x_j=x_k\}: 1\leq i,j,k\leq n\, \textnormal{ with distinct } i,j,k\big\}.\]
Suppose that $x\in\sn$ is a point in this intersection. Then either all coordinates of $x$ are equal, or at least two are different, say $x_1$ and $x_2$. But then all other coordinates have to be zero. At least one of $x_1$ and $x_2$ has to be nonzero, say $x_1$. Since $x\in\{x_1=0\}\cup\{x_2=x_3\}$, we have $x_2=0$ and thus $x_1=\pm1$. This implies that $\mu(T^n,\cdot)$ is concentrated at the points $\pm e_i$ and $\pm\langle e_1+\cdots+e_n\rangle$. So also in this case $\mu(T^n,\cdot)$ is discrete.

By Lemmas \ref{l1} and \ref{simpl22} we infer that there exist constants $c_1,c_2\geq 0$ such that for all $s>0$ 
\[\mu(sT^n,\cdot)=c_1S_p(sT^n,\cdot)+c_2S_p(-sT^n,\cdot).\]
By the simplicity of $\mu$ we also have $\mu(\{o\},\cdot)=0$. Lemma \ref{ext} therefore concludes the proof.
\end{proof}

\section*{Acknowledgments}
The work of the authors was supported by Austrian Science Fund (FWF) Project P23639-N18.

\bibliography{proposal} 

\bibliographystyle{dissbib}

\noindent
Christoph Haberl \\
University of Salzburg \\
Hellbrunnerstra\ss{}e 34 \\
5020 Salzburg, Austria \\
christoph.haberl@sbg.ac.at \\

\noindent
Lukas Parapatits \\
Vienna University of Technology \\
Institute of Discrete Mathematics and Geometry \\
Wiedner Hauptstra\ss{}e 8-10/1046 \\
1040 Vienna, Austria \\
lukas.parapatits@tuwien.ac.at

\end{document}